\documentclass[a4paper,12pt]{amsart}
\usepackage[utf8]{inputenc}
\usepackage{amsmath}
\usepackage{amsfonts}
\usepackage{amsthm}
\usepackage{amssymb}
\usepackage[mathscr]{euscript}
\usepackage{xcolor}
\usepackage{enumerate}
\usepackage{tikz}
\usepackage{url}
\usepackage{soul}

\usepackage[
margin=1in,
marginpar=2cm,
includefoot,
footskip=30pt,
]{geometry}

\newtheorem{theorem}[equation]{Theorem}
\newtheorem{lemma}[equation]{Lemma}
\newtheorem{proposition}[equation]{Proposition}
\newtheorem{corollary}[equation]{Corollary}

\theoremstyle{definition}
\newtheorem{definition}[equation]{Definition}

\theoremstyle{remark}
\newtheorem{example}[equation]{Example}
\newtheorem{remark}[equation]{Remark}

\numberwithin{equation}{section}

\newcommand{\osf}{{\normalfont \textsf{X}}}
\newcommand{\osfY}{{\normalfont \textsf{Y}}}
\newcommand{\gosf}[1]{\textsf{X}_{#1}}

\newcommand{\lang}{\CL_{\osf}}
\newcommand{\algshift}{\CA_R(\osf)}
\newcommand{\ualgshift}{\TCA_R(\osf)}

\newcommand{\ucsalgshift}{\TCO_\osf}
\newcommand{\ucsalgshiftY}{\TCO_\osfY}

\newcommand{\calgshift}{\CO_{\osf}}
\newcommand{\ucalgshift}{\TCO_{\osf}}

\newcommand{\alf}{\mathscr{A}}

\newcommand{\N}{\mathbb{N}}
\newcommand{\F}{\mathbb{F}}
\newcommand{\C}{\mathbb{C}}

\newcommand{\CA}{\mathcal{A}}
\newcommand{\CB}{\mathcal{B}}

\newcommand{\CE}{\mathcal{E}}
\newcommand{\CF}{\mathcal{F}}
\newcommand{\CG}{\mathcal{G}}

\newcommand{\CL}{\mathcal{L}}
\newcommand{\CO}{\mathcal{O}}
\newcommand{\CZ}{\mathcal{Z}}

\newcommand{\Img}{\operatorname{Im}}

\newcommand{\tlspace}{(\CE,\CL,\TCB)}

\newcommand{\nn}{\mathbb{N}}

\newcommand{\scj}{\subseteq}

\newcommand{\eword}{\omega}

\newcommand{\vecspan}{\operatorname{span}}
\newcommand{\cvecspan}{\overline{\vecspan}}

\newcommand{\Lc}{\operatorname{Lc}}
\newcommand{\dom}{\operatorname{dom}}

\newcommand{\TCA}{\widetilde{\CA}}
\newcommand{\TCB}{\mathcal{U}}

\newcommand{\TCO}{\widetilde{\CO}}

\newcommand{\HTCB}{\widehat{\TCB}}

\newcommand{\hsig}{\widehat{\sigma}}
\newcommand{\hh}{\widehat{h}}

\newcommand{\tauh}{\widehat{\tau}}
\newcommand{\varphih}{\widehat{\varphi}}

\title{C*-Algebras of one-sided subshifts over arbitrary alphabets}
\author[G. Boava]{Giuliano Boava}
\author[G.G. de Castro]{Gilles G. de Castro}
\author[D. Gonçalves]{Daniel Gonçalves}
\address[Giuliano Boava, Gilles G. de Castro and Daniel Gonçalves]{Departamento de Matem\'atica, Universidade Federal de Santa Catarina, 88040-970 Florian\'opolis SC, Brazil. }
\email{g.boava@ufsc.br \\ gilles.castro@ufsc.br \\ daemig@gmail.com}
\author[D.W. van Wyk]{Daniel W. van Wyk}
\address[Daniel W. van Wyk]{Department of Mathematics, Fairfield University, Fairfield, CT 06824 USA, and  \\
Department of Mathematics and Applied Mathematics, University of the Free State, Park West, Bloemfontein, 9301, South Africa.}
\email{dvanwyk@fairfield.edu }

\begin{document}

\keywords{Shift spaces, subshifts, infinite alphabets, C*-algebras, conjugacy, groupoids}
\subjclass[2020]{Primary: 46L55, 37B10. Secondary: 37A55, 37B05, 46L35, 22A22}

\thanks{The second and third authors were partially supported by Capes-Print Brazil. The third author was partially supported by Conselho Nacional de Desenvolvimento Cient\'ifico e Tecnol\'ogico (CNPq) - Brazil, and Funda\c{c}\~ao de Amparo \`a Pesquisa e Inova\c{c}\~ao do Estado de Santa Catarina (FAPESC).}

\begin{abstract}
We associate a C*-algebra $\widetilde{\mathcal{O}}_{\textsf{X}}$ with a subshift over an arbitrary, possibly infinite, alphabet. We show that $\widetilde{\mathcal{O}}_{\textsf{X}}$ is a full invariant for topological conjugacy of the subshifts of Ott, Tomforde, and Willis. When the alphabet is countable, we show that $\widetilde{\mathcal{O}}_{\textsf{X}}$ is an invariant for isometric conjugacy of subshifts with the product metric. For a suitable partial action associated with a subshift over a countable alphabet, we show that $\widetilde{\mathcal{O}}_{\textsf{X}}$ is also an invariant for continuous orbit equivalence. Additionally, we give a concrete way to compute the K-theory of $\widetilde{\mathcal{O}}_{\textsf{X}}$ and illustrate it with two examples.
\end{abstract}

\maketitle

\section{Introduction}
The study of subshifts is at the heart of symbolic dynamics. Describing various types of equivalences between subshifts is one of the main arteries of research in the field. For example, for subshifts of finite type, conjugacy is equivalent to strong shift equivalence. However, it is not known for which classes of subshifts conjugacy is equivalent to shift equivalence, which is easier to decide. This question is usually referred to as the Williams' problem, see \cite{CGGH, MR0331436, Williams}.

Recent work allows a complete recast of Williams' problem in operator algebras since isomorphism and Morita equivalence of two graph C*-algebras given by shifts of finite type are connected to equivalences of the subshifts, such as topological conjugacy and shift equivalence (see, for instance, \cite{CDE} for recent results and further details). 
These results follow the celebrated interaction between dynamics and C*-algebras, which includes the work of Giordano-Matui-Putnam-Skau on describing orbit equivalence of Cantor minimal system via C*-algebraic invariants, \cite{GMPS, GPS}, the seminal work of Cuntz-Krieger on C*-algebras associated with shifts of finite type (given by matrices), \cite{CuntzKrieger}, Matsumoto and Matui's work on orbit equivalence and continuous orbit equivalence of topological Markov chains \cite{MatsumotoOrbit,MatuiMatsumoto}, Carlsen's work on subshift algebras over finite alphabets, \cite{CarlsenShift}, and more recently the work by Brix-Carlsen on the description of conjugacy of subshifts over finite alphabets via an isomorphism of associated algebras, \cite{BrixCarlsen}, to mention a few.

The aim of this paper is in the same spirit as the previous paragraph. We connect the theory of subshifts over infinite alphabets with operator algebras. For an infinite alphabet, there is no standard definition of a subshift. The usual approach for the full shift is to equip a countable alphabet $\alf$ with the discrete topology. Then consider the product $\prod_{n\in\nn} \alf$ with the product topology. This topology is induced by the product metric
\[d((x_n),(y_n))= \frac{1}{2^k},\] 
where $k$ is such that $x_i=y_i$ for $i=0,\ldots k-1$ and $x_k\neq y_k$. In this setting, a subshift $\osf$ is defined as a shift-invariant closed subspace. Alternatively, a subshift can be defined in terms of a family of forbidden blocks (see \cite{DGM, MR} and Section~\ref{symbolic} for details). With the relative topology, the space obtained is Polish, but might not even be locally compact. The lack of local compactness is an obstacle to using analytical techniques to study such subshifts.  

In \cite{OTW}, Ott, Tomforde, and Willis propose an alternative definition of a subshift over an infinite alphabet. They first consider a compact Hausdorff space as a new version for the full subshift. Then, they define a subshift as a shift-invariant closed subspace with an extra condition, which they call the infinite extension property (see \cite[Definition~3.3]{OTW}). We call these OTW-subshifts. Since OTW-subshifts are compact Hausdorff spaces, analytical tools are more readily available. Moreover, starting with a subshift $\osf$, there is a corresponding OTW-subshift $\osf^{OTW}$. In Theorem~\ref{aguaceiro}, we show that isometric conjugacy for subshifts implies conjugacy of the correspondent OTW-subshifts.

Recently, the authors of this paper introduced algebras associated with a subshift in a purely algebraic setting, see \cite{BCGW22}. These algebras deepened the connection between OTW-subshifts and noncommutative algebras, which in the original work of Ott-Tomforde-Willis is restricted to the case of edge subshifts associated with a certain class of graphs. In \cite{BCGW22}, the authors describe conjugacy of OTW-subshifts in terms of isomorphisms of algebras associated with $\osf$.

We do not know of a C*-algebra associated with a general subshift over an infinite alphabet that forms an invariant for notions of equivalence of subshifts. In this paper we remedy this. We overcome the lack of local compactness of a subshift $\osf$ by using combinatorial aspects of it to define a C*-algebra $\ucalgshift$ associated with $\osf$. We apply the results of \cite{TopConjLocHom} to show that $\ucalgshift$ together with additional data is a complete invariant for conjugacy of OTW-subshifts (Theorem~\ref{stariso}). This is done by identifying an OTW-subshift with the spectrum of a certain commutative C*-subalgebra of $\ucalgshift$. Moreover, using the connection between $\osf$ and $\osf^{OTW}$, we obtain one of the main results of this paper, which says that our C*-algebras form an invariant for isometric conjugacy of subshifts over countable alphabets (Theorem~\ref{invariant}).

The C*-algebras we introduce in this paper are the analytical counterparts of the algebras introduced in \cite{BCGW22}. We use the results of \cite{BCGW22} and of \cite{EGG} on core subalgebras to obtain results about the structure of $\ucalgshift$. In particular, we show that $\ucalgshift$ can be described using labelled spaces, groupoids and partial actions.

The K-theory of C*-algebras provides a standard invariant for isomorphism of C*-algebras. The description of a unital subshift C*-algebra $\ucalgshift$ using labelled spaces and \cite[Theorem~4.4]{MR3614028} provide a concrete way to compute the K-theory of $\ucalgshift$. We illustrate this by computing the K-theory of two examples. In the first, we use K-theory to show that there is a C*-algebra of a graph that is not isomorphic to any of our C*-algebras associated with the edge subshift of the graph. In the second example, we compute the K-theory of the C*-algebra of a 2-step subshift that is not conjugate to any 1-step subshift (example taken from \cite{DDStep}). In particular, it is not conjugate to any graph edge subshift. Lastly, from our theorem about isometric conjugacy, we deduce that the K-theory of subshift C*-algebras is an invariant for isometric conjugacy of subshifts.

We use our groupoid model for $\ucalgshift$ to prove Theorem~\ref{stariso}, which gives a full invariant for topological conjugacy of OTW-subshifts. The unit space of this groupoid is homeomorphic to the spectrum of another commutative C*-subalgebra of $\ucalgshift$. Under an extra condition, isomorphism of subshift C*-algebras preserving these commutative C*-subalgebras is described both as an isomorphism of the corresponding groupoids and as a continuous orbit equivalency between certain topological partial actions (Theorem~\ref{diadesol}).

The paper is organized as follows. Section~\ref{symbolic} contains basic elements of symbolic dynamics and the definitions of subshifts and OTW-subshifts.

In Section~\ref{c-subshift}, we define two C*-algebras, $\ucalgshift$ and $\calgshift$, associated with a subshift $\osf$, where $\ucalgshift$ is unital and $\calgshift$ is not necessarily unital. The latter algebra is mostly used to compare with non-unital graph and ultragraph C*-algebras. The main focus of this paper is on $\ucsalgshift$. We show that the purely algebraic subshift algebras of \cite{BCGW22} are core subalgebras of our C*-algebras. We realize $\ucalgshift$ and $\calgshift$ as C*-algebras of labelled spaces. Finally, we prove a gauge invariance uniqueness theorem. 

In Section~\ref{examples}, we provide examples of known C*-algebras that can be seen as C*-algebras of a subshift. First, we show that the C*-algebra of a graph is isomorphic to the edge subshift C*-algebra, provided the graph has no sinks and no vertex that is simultaneously a source or an infinite emitter. Our second example shows that for ultragraphs with only regular vertices, the C*-algebra of the ultragraph is the C*-algebra of the one-sided edge subshift of the ultragraph. Our last example of this section shows that every unital Exel-Laca algebra is isomorphic to the unital C*-algebra of a subshift.

We dedicate Section~\ref{k-theory} to the K-theory of the unital subshift C*-algebra $\ucalgshift$. Adapting \cite[Theorem~4.4]{MR3614028} to $\ucalgshift$, we obtain a concrete way to compute the K-theory of $\ucalgshift$ and apply it to two examples. 

Finally, Section~\ref{dynamical} contains the main results of this paper, Theorems~\ref{diadesol}, \ref{stariso}, \ref{aguaceiro} and \ref{invariant}, which connect the dynamical and algebraic aspects as outlined above.

\section{Symbolic Dynamics}\label{symbolic}
In this paper, we consider $\N=\{0,1,2,3,\ldots\}$. 
Let $\alf$ be a non-empty set, called an \emph{alphabet}. The \emph{shift map} on $\alf^\N$ is the map $\sigma: \alf^\N\to \alf^\N$ given by $\sigma(x)=(y_n)$, where $x=(x_n)$ and  $y_n=x_{n+1}$. Elements of $\alf^*:=\bigcup_{k=0}^\infty \alf^k$ are called \emph{blocks} or \emph{words}, and $\omega$ stands for the empty word. Set $\alf^+=\alf^*\setminus\{\eword\}$. Given $\alpha\in\alf^*\cup\alf^{\N}$, $|\alpha|$ denotes the length of $\alpha$ and for $1\leq i,j\leq |\alpha|$, we define $\alpha_{i,j}:=\alpha_i\cdots\alpha_j$ if $i\leq j$, and $\alpha_{i,j}=\eword$ if $i>j$. If $\beta\in\alf^*$, then $\beta\alpha$ denotes concatenation of $\beta$ and $\alpha$. A subset $\osf \subseteq \alf^\N$ is \emph{invariant} for $\sigma$ if $\sigma (\osf)\subseteq \osf$. For an invariant subset $\osf \subseteq \alf^\N$, we define $\CL_n(\osf)$ as the set of all words of length $n$ that appear in some sequence of $\osf$, that is, $$\CL_n(\osf):=\{(a_0\ldots a_{n-1})\in \alf^n:\ \exists \ x\in \osf \text{ s.t. } (x_0\ldots x_{n-1})=(a_0\ldots a_{n-1})\}.$$ Clearly, $\CL_n(\alf^\N)=\alf^n$ and we always have that $\CL_0(\osf)=\{\omega\}$. The \emph{language} of $\osf$ is the set  $$\lang:=\bigcup_{n=0}^\infty\CL_n(\osf),$$ 
consisting of all finite words that appear in some sequence of $\osf$.

Given $F\subseteq \alf^*$, we define the \emph{subshift} $\osf_F\subseteq \alf^\N$ as the set of all sequences $x$ in $\alf^\N$ such that no word of $x$ belongs to $F$. Usually, the set $F$ will not play a role, so we will say $\osf$ is a subshift with the implication that $\osf=\osf_F$ for some $F$. We point out that subshifts are also called shift spaces in the literature. Given a subshift $\osf$ over an alphabet $\alf$ and $\alpha,\beta\in \lang$, we define \[C(\alpha,\beta):=\{\beta x\in\osf:\alpha x\in\osf\}.\] In particular, $Z_{\beta}:=C(\eword,\beta)$ is called the \emph{cylinder set} of $\beta$,  and $F_{\alpha}=C(\alpha,\eword)$ the \emph{follower set} of $\alpha$.

We briefly recall the construction of OTW-subshifts (OTW stands for Ott-Tomforde-Willis), see \cite{OTW}. Let $\alf$ be an alphabet and define $\tilde \alf:=\alf$ if $\alf$ is finite and $\tilde \alf:=\alf\cup\{\infty\}$ if $\alf$ is infinite, where $\infty$ is a new symbol not in $\alf$. Define \[ \Sigma_\alf := \{(x_i)_{i\in \N}\in \tilde \alf^\N: x_i = \infty \text{ implies }x_{i+1}= \infty \}, \] $\Sigma_\alf^{\text{inf}} := \alf^\N$ and $\Sigma_\alf^{\text{fin}} := \Sigma_\alf\setminus\Sigma_\alf^{\text{inf}}$. When the alphabet is infinite, the set $\Sigma_\alf^{\text{fin}}$ is identified with all finite sequences over $\alf$ via the identification
\begin{equation*}\label{xfin}
 (x_0x_1\dots x_k\infty\infty\infty\dots)\equiv (x_0x_1\ldots x_k).    
\end{equation*}
The sequence $(\infty \infty\infty\ldots)$ is denoted by $\vec{0}$. Let $F \subseteq  \alf^*$. We define
\begin{align*}
\osf^{inf}_F &:= \{ x \in \alf^\N : \text{ no block of $x$ is in $F$} \}, \text{ and } \\
\osf^{fin}_F &:= \{ x \in \Sigma_\alf^{\text{fin}} : \text{ there are infinitely many $a \in \alf$ for which}\\
& \qquad \qquad \qquad \qquad \text{there exists $y \in \alf^\N$ such that $xay \in \osf^{inf}_F$}\}. \\
\end{align*}
The \emph{OTW-subshift associated with $F$} is $\osf_F^{OTW}:= \osf^{inf}_F \cup \osf^{fin}_F$ together with a shift map $\sigma:\osf_F^{OTW}\to \osf_F^{OTW}$ defined by
\[\sigma(x)=\begin{cases}
x_1x_2\ldots, & \text{if }x=x_0x_1x_2\ldots\in\osf_F^{inf} \\
x_1\ldots x_{n-1}, & \text{if }x=x_0\ldots x_{n-1}\in\osf_F^{fin}\text{ and }n\geq 2 \\
\vec{0}, & \text{if }x=x_0\in\osf_F^{fin}\text{ or }x=\vec{0}.
\end{cases}
\]
In the case that $F=\emptyset$, we have that $\osf_F^{OTW}=\Sigma_\alf$, which we call the \textit{OTW full shift}. Notice that $\osf^{inf}_F$ coincides with the subshift $\osf_F$ defined above. As in the case of a subshift, we omit the subscript $F$ and write $\osf^{OTW}$ for $\osf^{OTW}_F$. Let $\osf^{OTW}$ be an OTW-subshift and $\alpha\in \lang$. Then the \emph{follower set} of $\alpha$ is defined as the set \[\CF(\alpha):=\{y\in \osf^{OTW}: \alpha y \in \osf^{OTW}\}. \] For a finite set $F\subseteq \alf$ and $\alpha\in  \lang$, we define the \emph{generalised cylinder set} as
\[ \CZ(\alpha,F):=\{y\in\osf^{OTW}:\ y_i=\alpha_i\ \forall \, 1\leq i\leq |\alpha|,\ y_{|\alpha|+1}\notin F\}. \]
To simplify notation, we denote  $\CZ(\alpha,\emptyset)$ by $\CZ(\alpha)$. Endowed with the topology generated by the generalised cylinders, $\osf^{OTW}$ is a compact totally disconnected Hausdorff space and, in this topology, the generalised cylinders are compact and open. The shift map is continuous everywhere, except possibly at $\vec{0}$ (if $\vec{0}\in\osf^{OTW}$). If the alphabet is finite, then $\osf^{OTW}$ is the usual subshift and the topology given by the generalised cylinders is the relative product topology (see \cite[Remark~2.26]{OTW}).

\section{C*-algebras of subshifts}\label{c-subshift}
In this section, we define a unital and a not necessary unital C*-algebra associated with a subshift. Fix a subshift $\osf$ over an alphabet $\alf$.

\subsection{Unital C*-algebras of subshifts}\label{unital}
As in \cite{BCGW22}, let $\TCB$ to be the Boolean algebra of subsets of $\osf$ generated by all $C(\alpha,\beta)$ for $\alpha,\beta\in\lang$, that is, each element of $\TCB$ is a finite union of elements of the form \[C(\alpha_1,\beta_1)\cap\ldots\cap C(\alpha_n,\beta_n)  \cap C(\mu_1,\nu_1)^c\cap \ldots \cap C(\mu_m,\nu_m)^c.\]

\begin{definition}\label{quente}
The \textit{unital subshift C*-algebra $\ucsalgshift$ associated with $\osf$} is the universal unital C*-algebra generated by projections $\{p_A: A\in\TCB\}$ and partial isometries $\{s_a: a\in\alf\}$ subject to the relations:
\begin{enumerate}[(i)]
    \item $p_{\osf}=1$, $p_{A\cap B}=p_Ap_B$, $p_{A\cup B}=p_A+p_B-p_{A\cap B}$ and $p_{\emptyset}=0$, for every $A,B\in\TCB$;
    \item $s_{\beta}s^*_{\alpha}s_{\alpha}s^*_{\beta}=p_{C(\alpha,\beta)}$ for all $\alpha,\beta\in\lang$, where $s_{\eword}:=1$ and, for $\alpha=\alpha_1\ldots\alpha_n\in\lang$, $s_\alpha:=s_{\alpha_1}\cdots s_{\alpha_n}$.
\end{enumerate}
\end{definition}

\begin{remark}
    In \cite{CarlsenShift}, Carlsen defines a C*-algebra for subshifts over finite alphabets and proves it has a universal property. The definition above adapts the universal property described by Carlsen and generalises his definition to include infinite alphabets.
\end{remark}

As in \cite[Remark 3.4]{BCGW22}, $s_\alpha^* s_\alpha = p_{C(\alpha,\omega)}=p_{F_\alpha}$ and $s_\beta s_\beta^* = p_{C(\omega,\beta)}= p_{Z_\beta}$ for all $\alpha,\beta \in \lang$.

The following result is proved in \cite[Proposition 3.6]{BCGW22} for the unital algebra of a subshift. Since the same proof holds in the C*-algebraic case, we omit it here.

\begin{proposition}\label{lemma.algebra.unital} In $\ucsalgshift$ the following hold:
\begin{enumerate}[(i)]
    \item $s_a^*s_b= \delta_{a,b} p_{F_a}$, for all $a,b\in \alf$;
    \item $s_\alpha^*s_\alpha$ and $s_\beta^*s_\beta$ commute for all $\alpha,\beta\in\lang$;
    \item $s_\alpha^*s_\alpha$ and $s_\beta s_\beta^*$ commute for all $\alpha,\beta\in\lang$;
    \item $s_\alpha s_\beta=0$ for all $\alpha,\beta\in\lang$ such that $\alpha\beta\notin\lang$;
    \item $\ucsalgshift$ is generated by the set $\{s_a, s_a^*: a\in\alf\}\cup\{1\}$.
\end{enumerate}
\end{proposition}

Next, we show that the unital algebra $\ualgshift$ defined in \cite{BCGW22} can be seen as a dense subalgebra of $\ucsalgshift$. For that, we first build a representation that shows that each projection $p_A$, with $A\in\TCB\setminus\{\emptyset\}$, is non-zero. Then, we consider the canonical gauge action on $\ucsalgshift$, in order to obtain a $\mathbb{Z}$-grading on $\ucsalgshift$.

\begin{proposition}\label{luigikart}
    Consider the family of operators $\{S_a\}_{a\in\alf}$ and $\{P_A\}_{A\in\TCB}$ on $\ell^2(\osf)$ defined by
    \[S_a(\delta_x)=\begin{cases}
    \delta_{ax}, & \text{if }x\in F_a \\
    0 & \text{otherwise;}
    \end{cases}\qquad
    P_A(\delta_x)=\begin{cases}
    \delta_x, & \text{if }x\in A\\
    0, & \text{otherwise,}
    \end{cases}\]
    where $\{\delta_x\}_{x\in\osf}$ is the canonical orthonormal basis of $\ell^2(\osf)$. Then, there exists a *-representation $\pi:\ucalgshift\to B(\ell^2(X))$ such that $\pi(s_a)=S_a$ for all $a\in\alf$ and $\pi(p_A)=P_A$ for all $A\in\TCB$.
\end{proposition}

\begin{proof}
It is straightforward to check that $P_A$ is a projection, $P_{\osf}=1$, $P_{A\cap B}=P_AP_B$, $P_{A\cup B}=P_A+P_B-P_{A\cap B}$ and $P_{\emptyset}=0$, for every $A,B\in\TCB$.

Notice that, for every $a\in\alf$, $S_a$ is an isometry between the subspaces $\cvecspan\{\delta_x:x\in F_{a}\}$ and $\cvecspan\{\delta_x:x\in Z_a\}$. In particular, $S_a$ is a partial isometry, whose adjoint is given by
\[S_a^*(\delta_x)=\begin{cases}
    \delta_{y}, & \text{if }x=ay\in Z_a \\
    0, & \text{otherwise.}
\end{cases}\]

Now, for $\alpha=\alpha_1\cdots\alpha_n\in\lang$, defining $S_{\alpha}:=S_{\alpha_1}\cdots S_{\alpha_n}$, we obtain similar formulas for $S_\alpha$ and $S^*_\alpha$. It follows immediately that $S_{\beta}S_{\alpha}^*S_{\alpha}S^*_{\beta}=P_{C(\alpha,\beta)}$.

The existence of the representation comes from the universal property of $\ucalgshift$.
\end{proof}

\begin{corollary}\label{mariowonder}
    For all nonempty $A\in\TCB$, we have that $p_A\neq 0$. 
\end{corollary}

\begin{proof}
    For $A\in\TCB$ such that $A\neq\emptyset$, there exists $x\in A$. Using $P_A$ as defined in Proposition~\ref{luigikart}, we have that $P_A(\delta_x)=\delta_x$, so that $P_A\neq 0$. It then follows that $p_A\neq 0$.
\end{proof}

We obtain an action of the circle $\mathbb{T}$ on $\ucsalgshift$ as follows. For each $z\in\mathbb{T}$, we define  $\gamma_z:\ucsalgshift\to\ucsalgshift$ on the generators by $\gamma_z(p_A)=p_A$ for all $A\in\TCB$, and $\gamma_z(s_a)=zs_a$ for all $a\in\alf$. For $\alpha=\alpha_1\cdots\alpha_n$ and $\beta=\beta_1\cdots\beta_m\in\lang$, we have that
\[zs_{\beta_1}\cdots zs_{\beta_m}(zs_{\alpha_n})^*\cdots(zs_{\alpha_1})^*zs_{\alpha_1}\cdots zs_{\alpha_n}(zs_{\beta_m})^*\cdots(zs_{\beta_1})^*=z^{m-n+n-m}p_{C(\alpha,\beta)}=p_{C(\alpha,\beta)}.\] Using the universal property of $\ucsalgshift$, we get that $\gamma_z$ is a *-homomorphism. It is clear that $\gamma_1$ is the identity,  $\gamma_{zw}=\gamma_z\circ\gamma_w$ for every $z,w\in\mathbb{T}$, and that $\gamma_z$ is an automorphism with $\gamma_z^{-1}=\gamma_{z^{-1}}$. We call the action $\gamma$ the gauge action and observe that it induces a $C^*$-grading on $\ucsalgshift$ by $\mathbb{Z}$, with fibers given by $$D_n:=\{x\in\ucsalgshift:\gamma_z(x)=z^nx\ \forall z\in\mathbb{T}\},$$ for each $n\in\mathbb{Z}$. 

\begin{remark}\label{algebraic.involution}
    In several of the following results, we consider an algebraic version $\TCA_{\C}(\osf)$ of $\ucsalgshift$ and its subalgebras, as defined in \cite[Definition~3.3]{BCGW22}. The algebras in \cite{BCGW22} are defined over a general ring $R$. When we specialize to the case $R=\C$, we can consider the natural involution on $\TCA_{\C}(\osf)$ defined by $(\lambda s_\alpha p_As_{\beta}^*)^*= \bar{\lambda} s_\beta p_As_{\alpha}^*$.
\end{remark}

In the next result, we show that there is an injective homomorphism from the unital subshift algebra $\TCA_{\C}(\osf)$ into $\ucsalgshift$.

\begin{proposition}\label{zelda}
    The canonical homomorphism $\TCA_{\C}(\osf) \to \ucsalgshift$ that sends each generator to its corresponding generator with the same name is injective and has a dense image.
\end{proposition}

\begin{proof}
    For $x=s_\alpha p_As_{\beta}^*\in\ucsalgshift$ and $z\in\mathbb{T}$, we have that $\gamma_z(x)=z^{|\alpha|-|\beta|}x$ so that $x\in D_{|\alpha|-|\beta|}$. By \cite[Proposition~3.9]{BCGW22}, the canonical homomorphism $\TCA_{\C}(\osf) \to \ucsalgshift$ is a $\mathbb{Z}$-graded ring homomorphism and by Corollary~\ref{mariowonder}, $\lambda p_A\neq0$ for every $\lambda\in \mathbb{C}\setminus\{0\}$ and $A\in\TCB\setminus\{\emptyset\}$. It follows from the Graded Uniqueness Theorem \cite[Corollary~3.15]{BCGW22} that the homomorphism is injective.
    
    The second part is due to the fact that the copy of $\TCA_{\C}(\osf)$ inside $\ucsalgshift$ contains the generators of $\ucsalgshift$.
\end{proof}

Let $\osf^{OTW}$ be as in Section~\ref{symbolic} and $\widehat{\TCB}$ the Stone dual of $\TCB$. To characterize $C(\osf^{OTW})$ and $C(\HTCB)$ we will use the results in \cite{BCGW22} and the theory of core subalgebras introduced in \cite{EGG}, which we briefly recall below.

\begin{definition}\cite[Definition~3.1]{EGG}
Let $A$ be a C*-algebra and $B\subseteq A$ a (not necessarily closed) *-subalgebra. We say that $B$ is a \emph{core subalgebra of $A$} when every representation of $B$ is continuous relative to the norm induced from $A$. By a representation of a *-algebra $B$ we mean a multiplicative, *-preserving, linear map $\pi:B\to {\CB}(H)$, where $H$ is a Hilbert space.
\end{definition}

\begin{remark}\label{r:core_subalg_univ}
    By \cite[Example~3.3(v)]{EGG}, whenever $A$ is a universal C*-algebra generated by generators and relations and $B$ is the free *-algebra with the same generators and relations, then the image of $B$ inside $A$ is a core subalgebra. In particular, this applies to both graph algebras, ultragraph algebras and subshift algebras.
\end{remark}

For examples of core subalgebras, we refer the reader to \cite{EGG}. The following result, which will be useful to us, is proved in \cite{EGG}.

\begin{proposition}\cite[Proposition~3.4]{EGG}\label{core iso}  Suppose that $A_1$ and $A_2$ are C*-algebras and that $B_i$ is a dense core subalgebra of $A_i$, for $i=1,2$. If $B_1$ and $B_2$ are isomorphic as *-algebras, then $A_1$ and $A_2$ are isometrically *-isomorphic.
\end{proposition}

\begin{lemma}\label{core projections}
    Let $A$ be a C*-algebra and $B$ a commutative *-algebra generated by projections of $A$. Then $B$ is a core subalgebra of $A$.
\end{lemma}

\begin{proof}
    Any non-zero element of $x\in B$ can be written as linear combination $x=\sum_{i=1}^n\lambda_i p_i$, where $\lambda_i\in\C\setminus\{0\}$ for all $i=1,\ldots,n$ and $p_1,\ldots,p_n$ are mutually orthogonal projections in $B$ (see for instance \cite[Lemma 4.1]{MR4343791}). Then, $\|x\|=\max_{i=1,\ldots,n}|\lambda_i|$ and, for any representation $\pi$ of $B$, we have that $\pi(p_1),\ldots,\pi(p_n)$ are mutually orthogonal projections so that $\|\pi(x)\|\leq \max_{i=1,\ldots,n}|\lambda_i|=\|x\|$. Therefore, $\pi$ is continuous with respect to the norm induced from $A$.
\end{proof}

For a topological space $X$, let $\Lc(X,\C)$ denote the locally constant functions from $X$ to $\C$ with compact support. By a Stone space, we mean a locally compact, totally disconnected Hausdorff space.

\begin{lemma}\label{core lc}
    Let $X$ be a Stone space. Then, $\Lc(X,\C)$  is a dense core subalgebra of $C_0(X)$.
\end{lemma}

\begin{proof}
    Since $\Lc(X,\C)$ is generated by the characteristic functions of compact-open sets, which are projections in $C_0(X)$, it follows from Lemma~\ref{core projections} that $\Lc(X,\C)$ is a core subalgebra of $C_0(X)$. As $X$ is a Stone space, the characteristic functions of compact-open sets generate $C_0(X)$ as a C*-algebra, which implies $\Lc(X,\C)$ is dense in $C_0(X)$.
\end{proof}

\begin{proposition}\label{nublado}
Let $\osf\subseteq \alf^\N$ be a subshift. Then, \[\cvecspan\{s_\alpha s_\alpha^*:\alpha \in \lang\}\cong C(\osf^{OTW}).\] Moreover,
\[\cvecspan\{s_\alpha p_A s_\alpha^*: A\in\TCB,\ \alpha \in \lang\}=\cvecspan\{p_A: A\in\TCB\}\cong C(\widehat{\TCB}).\]
\end{proposition}

\begin{proof}
    By \cite[Proposition~3.17]{BCGW22} and Remark~\ref{algebraic.involution}, $\vecspan\{s_\alpha s_\alpha^*:\alpha \in \lang\}$ and $\Lc(\osf^{OTW},\C)$ are *-isomorphic. Observe that $\Lc(\osf^{OTW},\C)$ is a dense core subalgebra of $C(\osf^{OTW})$ by Lemma~\ref{core lc}, since $\osf^{OTW}$ is a Stone space. Also, by Proposition~\ref{zelda} and Lemma~\ref{core projections}, $\vecspan\{s_\alpha s_\alpha^*:\alpha \in \lang\}$ is a dense core subalgebra of $\cvecspan\{s_\alpha s_\alpha^*:\alpha \in \lang\}$. Thus, the first isomorphism follows from Proposition~\ref{core iso}.

    For the second isomorphism, we use the same idea and apply \cite[Proposition~3.19]{BCGW22} in place of \cite[Proposition~3.17]{BCGW22}.
\end{proof}

In \cite{MR3614028} a normal labelled space $\tlspace$ is associated with a subshift $\osf$ as follows: the graph $\CE$ is given by $\CE^0=\osf$, $\CE^1=\{(x,a,y)\in\osf\times\alf\times\osf: x=ay\}$, $s(x,a,y)=x$ and $r(x,a,y)=y$. The labelling map is given by $\CL(x,a,y)=a$, and the accommodating family $\TCB$ is the Boolean algebra defined above. Then, the triple $\tlspace$ is a normal labelled space \cite[Lemma 5.5]{MR3614028}.

For $A\in \TCB$, define $\CL(A\CE^1)=\{a\in\alf: Z_a\cap A\neq\emptyset\}$ and let $\TCB_{reg} = \{A\in \TCB :  0<|\CL(A\CE^1)|<\infty\}$. Then, $A\in\TCB_{reg}$ if and only if $Z_a\cap A\neq\emptyset$ for finitely many $a\in\alf$. Note that, because $\CE$ has no sinks, the definitions of $\CL(A\CE^1)$ and $\TCB_{reg}$ above coincide with the general definitions of these sets for labelled spaces (see \cite{BCGW22}). In particular, if the alphabet is finite, then $\TCB_{reg}=\TCB$. Also, since $\osf=\bigsqcup_{a\in\alf}Z_a$, we have that if $A\in\TCB$ then $A=\bigsqcup_{a\in\CL(A\CE^1)}Z_a\cap A$.

For $\alpha,\beta\in\lang$ such that $\beta\neq\eword$ and $a\in\alf$, we have from \cite[Equation (6)]{MR3614028} that the relative range of the set $C(\alpha,\beta)$ is given by 

\begin{equation}\label{eq:rel.range}
r(C(\alpha,\beta),a)=\begin{cases}
C(a,\eword)\cap C(\alpha,\beta_2\ldots\beta_{|\beta|}) & \text{if } \beta=a\beta_2,\ldots,\beta_{|\beta|}, \\
\emptyset & \text{otherwise}.
\end{cases}    
\end{equation}

Also,
\[r(C(\alpha,\eword),a)=\begin{cases}
C(\alpha a,\eword) & \text{if }\alpha a\in \lang, \\
\emptyset & \text{if }\alpha a\notin \lang.
\end{cases}\]
More generally, \[r(A,\alpha)=\{x\in\osf: \alpha x\in A\}.\] This last equality implies that $r(\alpha)=r(\osf,\alpha)=F_{\alpha}=C(\alpha,\eword)$.

Following the same line of thought as in \cite[Theorem 3.12]{BCGW22}, we obtain that the subshift C*-algebra may be realized as the C*-algebra of the labelled space $\tlspace$. We state this as a theorem below.

\begin{theorem}\label{melaoflutuante}
Let $\osf$ be a subshift and let $\tlspace$ be the labelled space defined above. Then, $\ucsalgshift\cong C^*\tlspace$.
\end{theorem}

Using Theorem~\ref{melaoflutuante} and \cite[Corollary~3.10]{MR3614028} we obtain the following.

\begin{corollary}[Gauge-invariance uniqueness theorem]\label{gut_unital}
    Let $\osf$ be a subshift, $B$ a C*-algebra and $\Phi:\ucsalgshift\to B$ a *-homomorphism. If $\Phi(p_A)\neq 0$ for all $A\in\TCB\setminus\{\emptyset\}$, and there exists a strongly continuous action $\beta:\mathbb{T}\to B$ such that $\Phi\circ\gamma_z=\beta_z\circ\Phi$ for all $z\in\mathbb{T}$, then $\Phi$ is injective.
\end{corollary}

\subsection{Non-unital C*-algebras of subshifts}
In this section, we define a possibly non-unital C*-algebra $\calgshift$ associated with a subshift $\osf$.

Let $\CB$ be the Boolean algebra of subsets of $\osf$ generated by all $C(\alpha,\beta)$, for $\alpha,\beta\in\lang$ not both simultaneously equal to $\eword$. In particular, we do not require that $\osf$ is a generator for the Boolean algebra $\CB$. In some cases $\CB$ and $\TCB$ coincide \cite[Example 4.16]{BCGW22}, and in others they do not \cite[Example 4.11]{BCGW22}. 

\begin{definition}\label{nonunitalalgebra}
The \emph{subshift C*-algebra} $\calgshift$ is the universal C*-algebra generated by projections $\{p_A: A\in\CB\}$ and partial isometries $\{s_a: a\in\alf\}$ subject to the relations:
\begin{enumerate}[(i)]
    \item $p_{A\cap B}=p_Ap_B$, $p_{A\cup B}=p_A+p_B-p_{A\cap B}$ and $p_{\emptyset}=0$, for every $A,B\in\CB$;
    \item $s_{\beta}s^*_{\alpha}s_{\alpha}s^*_{\beta}=p_{C(\alpha,\beta)}$ for all $\alpha,\beta\in\lang\setminus\{\eword\}$, where for $\alpha=\alpha_1\ldots\alpha_n\in\lang \setminus \{\omega\}$, $s_\alpha:=s_{\alpha_1}\cdots s_{\alpha_n}$
    \item $s_\alpha^* s_\alpha = p_{C(\alpha,\omega)}$ for all $\alpha \in \lang \setminus \{\omega\}$;
    \item $s_\beta s_\beta^* = p_{C(\omega,\beta)}$ for all $\beta \in\lang \setminus \{\omega\}$.
\end{enumerate}
\end{definition}

\begin{remark}
Similar to \cite[Proposition~4.8]{BCGW22}, the C*-algebra $\calgshift$ coincides with $\ucalgshift$ when it is unital, and its unitization coincides with $\ucalgshift$ when it is not unital. 
\end{remark}

\begin{remark}
Propositions~\ref{lemma.algebra.unital}, \ref{luigikart}, \ref{zelda}, Theorem~\ref{melaoflutuante}, Corollaries~\ref{mariowonder} and \ref{gut_unital} hold for non-unital C*-algebras of subshifts. The proofs are similar and we omit them.
\end{remark}

\section{Examples}\label{examples}
In this section, we realize several known algebras as subshift algebras.

\subsection*{Graphs C*-algebras:}
Let $\CE=(\CE^0,\CE^1,r,s)$ be a directed graph, and let $C^*(\CE)$ denote the graph C*-algebra of $\CE$. The one-side edge subshift associated with $\CE$ is the set of all infinite paths, which is the subshift over the alphabet $\alf=\CE^1$ given by the family of forbidden words $\{ef\in\alf^2:r(e)\neq s(f)\}$.

\begin{proposition}\label{LPA}
Let $\CE$ be a graph with no sinks and with no vertex that is simultaneously a source and an infinite emitter. Let $\osf$ be the associated one-sided edge subshift of $\CE$. Then, $\calgshift\cong C^*(\CE)$.
\end{proposition}
\begin{proof} 
By Remark~\ref{r:core_subalg_univ}, the Leavitt path algebra $L_{\mathbb{C}}(\CE)$ is a core subalgebra of $C^*(\CE)$, and similarly the algebraic subshift algebra $\CA_{\C}(\osf)$ (see \cite[Definition 4.1]{BCGW22}) is a core subalgebra of $\calgshift$. By \cite[Proposition 4.10]{BCGW22}, we have that $L_{\mathbb{C}}(\CE) \cong \CA_{\C}(\osf)$. Notice that $(s_e)^*= s_e^*$ and  $(p_A)^*=p_A$ defines an involution on  $\CA_{\C}(\osf)$, and then the isomorphism $L_{\mathbb{C}}(\CE) \cong \CA_{\C}(\osf)$ of \cite[Proposition 4.10]{BCGW22} is *-preserving. Hence, Proposition~\ref{core iso} implies that $\calgshift\cong C^*(\CE)$.
\end{proof}

\begin{example}\label{theorPropfail}
Let $\CE$ be the graph such that $\CE^0=\{v,w\}$, $\CE^1=\{e_n\}_{n\in\nn}\cup\{f\}$, $s(e_n)=v$ and $r(e_n)=w=s(f)=r(f)$ for all $n\in\nn$. It was shown in \cite[Example 4.11]{BCGW22} that $L_R(\CE)$ is not isomorphic to $\algshift$ because the first is unital, whereas the latter is not. The same argument can be applied in the C*-algebraic setting. It is natural to compare the unital subshift algebra $\ucsalgshift$ with $C^*(\CE)$. In this case, considering the elements $q_v:=s_f^*s_f$, $q_w:=1-q_v$, $t_{e_n}:=s_{e_n}$ for $n\in\nn$ and $t_f:=s_f$ inside $\ucsalgshift$, we obtain a surjective *-homomorphism $\Phi:C^*(\CE)\to\ucsalgshift$. Using the gauge-invariant uniqueness theorem for graph C*-algebras, it follows that $\Phi$ is injective, and hence, an isomorphism.
\end{example}

The isomorphism in Example~\ref{theorPropfail} is not guaranteed in general. In Example~\ref{chuva}, we construct a graph $\CE$ with finitely many vertices and use K-theory to show that $C^*(\CE)$ is not isomorphic to $\ucsalgshift$.

\subsection*{Ultragraph C*-algebras:}
We refer the reader to \cite[Definition 4.12]{BCGW22} for the definition of an ultragraph. Fix an ultragraph $\mathcal{G}=(G^0, \mathcal{G}^1, r,s)$, and let $C^*(\mathcal G)$ denote the ultragraph C*-algebra associated with $\mathcal{G}$ (see \cite{MR2050134}). 

The associated one-sided edge subshift $\osf_\CG$ of $\CG$ is the set of all infinite paths. This is the subshift over the alphabet $\alf=\CG^1$ given by the family of forbidden words $\{ef\in\alf^2:s(f)\notin r(e)\}$. We show that a certain class of ultragraph C*algebras can be realized as subshift C*-algebras. 

\begin{proposition}\label{LPAU}
Let $\mathcal G$ be an ultragraph such that every vertex is regular, that is, $0<|s^{-1}(v)|<\infty$ for all $v\in G^0$, and let $\osf_\CG$ be its one-sided edge subshift. Then $C^*(\mathcal G)\cong \mathcal{O}_{\osf_\CG}$.
\end{proposition}
\begin{proof}
By Remark~\ref{r:core_subalg_univ}, the ultragraph Leavitt path algebra $L_{\mathbb{C}}(\mathcal G)$ is a core subalgebra of $C^*(\mathcal G)$. Remark~\ref{r:core_subalg_univ} also implies that $\CA_{\C}( \osf_\CG)$ is core subalgebra of  $\mathcal{O}_{\osf_\CG}$. Hence, \cite[Proposition 4.8]{BCGW22} and  Proposition~\ref{core iso} imply that $C^*(\mathcal G)\cong \mathcal{O}_{\osf_\CG}$.
\end{proof}

\subsection*{Unital Exel-Laca algebras:}
Let $I$ be a set of indices and $A=(A_{ij})_{i,j\in I}$ a $\{0,1\}$-matrix with no rows identically zero. We let $\gosf{A}=\{(x_n)\in I^\nn: A_{x_nx_{n+1}}=1\text{ for all }n\in\nn\}$. The unital Exel-Laca algebra $\TCO_A$ is the unital C*-algebra generated by partial isometries $\{S_i\}_{i\in I}$ such that
\begin{itemize}
    \item[(EL1)] $S_i^*S_iS_j^*S_j=S_j^*S_jS_i^*S_i$, for all $i,j\in I$;
    \item[(EL2)] $S_iS_i^*S_jS_j^*=0$, whenever $i\neq j$;
    \item[(EL3)] $S_i^*S_iS_jS_j^*=A_{ij}S_jS_j^*$ for all $i,j\in I$;
    \item[(EL4)] for all $X,Y\scj I$ finite such that
    \[A(X,Y,j):=\prod_{x\in X}A_{xj}\prod_{y\in Y}(1-A_{yj})\]
    is zero for all but a finite number of $j$'s, we have that,
    \[\prod_{x\in X}S_x^*S_x\prod_{y\in Y}(1-S_y^*S_y)=\sum_{j\in I} A(X,Y,j)S_jS_j^*.\]
\end{itemize}

\begin{proposition}
    In the context above, $\TCO_A\cong\TCO_{\osf_A}$.
\end{proposition}

\begin{proof}
We show that the elements $\{s_i\}_{i\in I}$ satisfy the relations defining $\TCO_A$. (EL1) and (EL2) follows from Proposition \ref{lemma.algebra.unital}. For (EL3), we observe that for $i,j\in I$ we have that $F_i\cap Z_j=\emptyset$ if $A_{ij}=0$, and $F_i\cap Z_j=Z_j$ if $A_{ij}=1$. Hence
\[s_i^*s_is_js_j^*=p_{F_i}p_{Z_j}=p_{F_i\cap Z_j}=A_{ij}p_{Z_j}=A_{ij}s_js_j^*.\]
For (EL4), let $X,Y\scj I$ be finite sets. Notice that
\[\bigcap_{x\in X} F_x\cap\bigcap_{y\in Y}F_y^c=\bigsqcup_{j:A(X,Y,j)=1}Z_j.\]
In particular, if $A(X,Y,j)$ is zero for all but a finite number of $j$'s, then the above union is finite and hence
\begin{align*}
    \prod_{x\in X}s_x^*s_x\prod_{y\in Y}(1-s_y^*s_y) &= \prod_{x\in X}p_{F_x}\prod_{y\in Y}(1-p_{F_y})\\
    &= \prod_{x\in X}p_{F_x}\prod_{y\in Y}(p_{F_y^c}) \\
    &= p_{\bigcap_{x\in X} F_x\cap\bigcap_{y\in Y}F_y^c} \\
    &= p_{\bigcup_{j:A(X,Y,j)=1}Z_j} \\
    &= \sum_{j:A(X,Y,j)=1}s_js_j^* \\
    &= \sum_{j\in I}A(X,Y,j)s_js_j^*.
\end{align*}

By the universal property of $\TCO_A$, there exists a $*$-homomorphism $\varphi:\TCO_A\to\TCO_{\gosf{A}}$, which is surjective by Proposition \ref{lemma.algebra.unital}. Injectivity follows from the gauge-invariant uniqueness theorem \cite[Theorem~2.7]{UniquenessExelLaca}.
\end{proof}

\begin{remark}
    An analogous result for the algebras that are not necessarily unital can be deduced from Proposition~\ref{LPAU} and \cite[Theorem~4.5]{MR2050134}.
\end{remark}

\section{K-Theory}\label{k-theory}
In \cite[Theorem~4.4]{MR3614028}, Bates, Carlsen and Pask describe the K-theory of a labelled space C*-algebra. Since every unital C*-algebra of a subshift may be realized as a labelled space C*-algebra by Theorem~\ref{melaoflutuante}, we can reformulate \cite[Theorem~4.4]{MR3614028} for C*-algebras of subshifts as follows.

\begin{theorem}\label{k-theory.theorem}
    Let $\osf$ be a one-sided subshift over an arbitrary alphabet $\alf$. Let $(1-\Phi): \operatorname{span}_{\mathbb{Z}}\left\{\chi_A: A \in \TCB_{reg}\right\} \rightarrow \operatorname{span}_{\mathbb{Z}}\left\{\chi_A: A \in \TCB\right\}$ be the linear map given by
$$
(1-\Phi)\left(\chi_A\right)=\chi_A-\sum_{a \in \CL\left(A \CE_{\osf}^1\right)} \chi_{r(A, a)}, \quad A \in \TCB_{reg} .
$$
Then $K_1(\ucsalgshift)$ is isomorphic to $\operatorname{ker}(1-\Phi)$ and there exists an isomorphism from $K_0(\ucsalgshift)$ to $\operatorname{coker}(1-\Phi)$ which maps $\left[S_\alpha^* S_\alpha\right]_0$ to $\chi_{r(\alpha)}+\operatorname{Im}(1-\Phi)$ for each $\alpha \in\lang$.
\end{theorem}

We present two applications of this result that illustrate the computation of K-theory for subshift C*-algebras. In the first example, we present a graph such its C*-algebra is not isomorphic to the corresponding edge subshift algebra. In the second example, we consider a subshift over an infinite alphabet that is not conjugate to any 1-step OTW-subshift, and thus does not fall under any of the examples of Section~\ref{examples}.

\begin{example}\label{chuva} Let $E$ be the graph with three vertices, say $v_1,v_2,v_3$, infinitely many edges from $v_1$ to $v_3$, infinitely many edges from $v_2$ to $v_3$, and a loop in $v_3$, see the picture below.
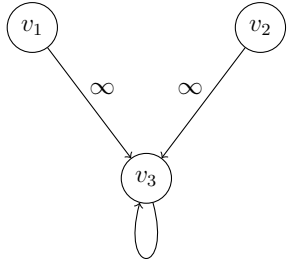
\begin{center}
\begin{tikzpicture}[every node/.style={scale=0.8}]
    \node[circle, draw] (v1) at (0,0) {$v_1$};
    \node[circle, draw] (v2) at (3,0) {$v_2$};
    \node[circle, draw] (v3) at (1.5, -2) {$v_3$};
    \draw[->, ] (v1) to node[above right] {$\!\!\infty$} (v3);
    \draw[->, ] (v2) to node[above] {$\!\!\!\!\!\!\infty$} (v3);
    \draw[->, loop below, looseness=15] (v3) to node[above] {} (v3);
\end{tikzpicture}
\end{center}
We use K-theory to show that the graph C*-algebra $C^*(E)$ is not isomorphic to the subshift algebra $\ucalgshift$, where $\osf$ is the edge subshift. 

Using \cite[Corollary~5.1]{MR3614028} (see also \cite[Theorem~3.1]{KtheoryGraph}) to compute the K-theory of $C^*(E)$, we obtain that $K_0(C^*(E)) \cong \mathbb{Z}^3$ and $K_1(C^*(E)) \cong \mathbb{Z}$.

We compute the K-theory of $\ucalgshift$. Denote the edges from $v_1$ to $v_3$ by $\{a_n\}_{n\in\mathbb{N}}$, from $v_2$ to $v_3$ by $\{b_n\}_{n\in\mathbb{N}}$ and the loop on $v_3$ by $c$. Then  $\osf = \{a_nc^\infty\}_{n\in\mathbb{N}} \cup \{b_nc^\infty\}_{n\in\mathbb{N}}\cup\{c^\infty\}$. It is easy to see that the elements in $\TCB_{reg}$ are the finite subsets of $\osf$, and the elements in $\TCB$ are the finite or cofinite subsets of $\osf$. Then, an element in $\vecspan_{\mathbb{Z}}\left\{\chi_A: A \in \TCB_{reg}\right\}$ is of the form
\[x = \sum_{n\in\mathbb{N}}^{finite}\lambda_{n}\chi_{\{a_nc^\infty\}} + \sum_{n\in\mathbb{N}}^{finite}\mu_{n}\chi_{\{b_nc^\infty\}} + \eta\chi_{\{c^\infty\}}.\]
Taking $\Phi$ as in Theorem~\ref{k-theory.theorem}, we have
\begin{align*}
(1-\Phi)(x) & =  (1-\Phi)\left(\sum_{n\in\mathbb{N}}^{finite}\lambda_{n}\chi_{\{a_nc^\infty\}} + \sum_{n\in\mathbb{N}}^{finite}\mu_{n}\chi_{\{b_nc^\infty\}} + \eta\chi_{\{c^\infty\}}\right) \\
& = \sum_{n\in\mathbb{N}}^{finite}\lambda_{n}(\chi_{\{a_nc^\infty\}} - \chi_{\{c^\infty\}}) + \sum_{n\in\mathbb{N}}^{finite}\mu_{n}(\chi_{\{b_nc^\infty\}} - \chi_{\{c^\infty\}}) + \eta(\chi_{\{c^\infty\}} - \chi_{\{c^\infty\}}) \\
& = \sum_{n\in\mathbb{N}}^{finite}\lambda_{n}(\chi_{\{a_nc^\infty\}} - \chi_{\{c^\infty\}}) + \sum_{n\in\mathbb{N}}^{finite}\mu_{n}(\chi_{\{b_nc^\infty\}} - \chi_{\{c^\infty\}}).
\end{align*}
Therefore, $x\in\ker(1-\Phi)$ if, and only if, $\lambda_{n}=0$ and $\mu_{n}=0$. It follows from Theorem~\ref{k-theory.theorem} that $K_1(\ucalgshift) \cong \ker(1-\Phi) = \vecspan_{\mathbb{Z}}\{\chi_{\{c^\infty\}}\} \cong \mathbb{Z}$. The above computations also show that
\[\Img(1-\Phi) = \left\{\sum_{n\in\mathbb{N}}^{finite}\lambda_{n}\chi_{\{a_nc^\infty\}} + \sum_{n\in\mathbb{N}}^{finite}\mu_{n}\chi_{\{b_nc^\infty\}} + \eta\chi_{\{c^\infty\}} \ \biggl| \ \eta = -\sum_{n\in\mathbb{N}}^{finite}\lambda_{n} - \sum_{n\in\mathbb{N}}^{finite}\mu_{n}\right\}.\]
Now, observe that an element in $\vecspan_{\mathbb{Z}}\left\{\chi_A: A \in \TCB\right\}$ is of the form 
\[y = \sum_{n\in\mathbb{N}}^{finite}\lambda_{n}\chi_{\{a_nc^\infty\}} + \sum_{n\in\mathbb{N}}^{finite}\mu_{n}\chi_{\{b_nc^\infty\}} + \eta\chi_{\{c^\infty\}} + \rho\chi_{\osf}.\]
Define the surjective map $\Psi:\vecspan_{\mathbb{Z}}\left\{\chi_A: A \in \TCB\right\}\to\mathbb{Z}^2$ by
\begin{align*}
\Psi(y) & = \Psi\left(\sum_{n\in\mathbb{N}}^{finite}\lambda_{n}\chi_{\{a_nc^\infty\}} + \sum_{n\in\mathbb{N}}^{finite}\mu_{n}\chi_{\{b_nc^\infty\}} + \eta\chi_{\{c^\infty\}} + \rho\chi_{\osf}\right) \\
& = \left(\eta + \sum_{n\in\mathbb{N}}^{finite}\lambda_{n} + \sum_{n\in\mathbb{N}}^{finite}\mu_{n}, \rho\right).
\end{align*}
Clearly, $y\in\ker\Psi$ if, and only if,  $\rho=0$ and $\displaystyle\eta = -\sum_{n\in\mathbb{N}}^{finite}\lambda_{n} - \sum_{n\in\mathbb{N}}^{finite}\mu_{n}$, i.e., $y\in\Img(1-\Phi)$. This shows that $\operatorname{coker}(1-\Phi) \cong \mathbb{Z}^2$. It follows from Theorem~\ref{k-theory.theorem} that $K_0(\ucalgshift) \cong \mathbb{Z}^2$.

Now, since $K_0(C^*(E))\cong \mathbb{Z}^3$ and $K_0(\ucalgshift) \cong \mathbb{Z}^2$, we conclude that $C^*(E)$ and $\ucalgshift$ are not isomorphic.
\end{example}

\begin{example}\label{tripla}
Take $\N$ as the alphabet and define $P$ as the following subset of $\N^3$: $$P:=\{(0,j,0),(j,0,j):j\in \N\}.$$ Let $F=\N^3\setminus P$. Then, by \cite{DDStep}, $\osf_F^{OTW}:=\osf^{OTW}$ is a 2-step OTW-subshift that is not conjugate to any 1-step OTW-subshift. Notice that $\osf:=\{ 0^\infty, (0j)^\infty, (j0)^\infty: j\in \N^*\}$. Consider the subsets of $\osf$: $A=\{0^\infty\}$, $B_j=\{(0j)^\infty\}$ for $j\geq1$, $C_j=\{(j0)^\infty\}$ for $j\geq1$, $U = \{(0j)^\infty, j\geq1\}$ and $V = \{(j0)^\infty, j\geq1\}$. We leave it to the reader to verify that 
\[\operatorname{span}_{\mathbb{Z}}\left\{\chi_A: A \in \TCB\right\} = \vecspan_{\mathbb{Z}}\{\chi_A, \chi_{B_j}, \chi_{C_j}, \chi_U, \chi_V\}\]
and 
\[\operatorname{span}_{\mathbb{Z}}\left\{\chi_A: A \in \TCB_{reg}\right\} = \vecspan_{\mathbb{Z}}\{\chi_A, \chi_{B_j}, \chi_{C_j}, \chi_U\}.\]
Let \[x=\sum_{i\in\mathbb{N}}^{finite}\lambda_i\chi_{B_i} + \sum_{j\in\mathbb{N}}^{finite}\mu_j\chi_{C_j} + \eta\chi_A + \rho\chi_U \in \operatorname{span}_{\mathbb{Z}}\left\{\chi_A: A \in \TCB_{reg}\right\}.\]
Taking $\Phi$ as in Theorem~\ref{k-theory.theorem}, we have that
\begin{align*}
(1 - \Phi)(x) & = (1-\Phi)\left(\sum_{i\in\mathbb{N}}^{finite}\lambda_i\chi_{B_i} + \sum_{j\in\mathbb{N}}^{finite}\mu_j\chi_{C_j} + \eta\chi_A + \rho\chi_U\right) \\
& = \sum_{i\in\mathbb{N}}^{finite}\lambda_i(\chi_{B_i} - \chi_{C_i}) + \sum_{j\in\mathbb{N}}^{finite}\mu_j(\chi_{C_j} - \chi_{B_j}) + \eta(\chi_A -\chi_A) + \rho(\chi_U - \chi_V) \\
& = \sum_{i\in\mathbb{N}}^{finite}(\lambda_i-\mu_i)\chi_{B_i} + \sum_{j\in\mathbb{N}}^{finite}(\mu_j-\lambda_j)\chi_{C_j} + \rho(\chi_U - \chi_V).
\end{align*}
This shows that $x\in\ker(1 - \Phi)$ if, and only if, $\rho=0$ and $\mu_j=\lambda_j$ for all $j\geq 1$. Therefore, by Theorem~\ref{k-theory.theorem},
\[K_1(\ucsalgshift) \cong \ker(1-\Phi) = \vecspan_{\mathbb{Z}}\{\chi_A, \chi_{B_j} + \chi_{C_j}\} \cong \bigoplus_{j\in\mathbb{N}}\mathbb{Z}.\]
The calculations above also show that $\Img(1-\Phi)$ is equal to
\[ \vecspan_{\mathbb{Z}}\left\{\sum_{i\in\mathbb{N}}^{finite}\lambda_i\chi_{B_i} + \sum_{j\in\mathbb{N}}^{finite}\mu_j\chi_{C_j} + \eta\chi_A + \rho\chi_U + \kappa\chi_V \ \Biggl| \ \eta = 0, \lambda_i + \mu_i = 0, \rho + \kappa = 0\right\}.\]
Now, define $\displaystyle\Psi:\vecspan_{\mathbb{Z}}\left\{\chi_A: A \in \TCB\right\}\to\bigoplus_{j\in\mathbb{N}}\mathbb{Z}$ by 
\[\Psi\left(\sum_{i\in\mathbb{N}}^{finite}\lambda_i\chi_{B_i} + \sum_{j\in\mathbb{N}}^{finite}\mu_j\chi_{C_j} + \eta\chi_A + \rho\chi_U + \kappa\chi_V\right) = (\eta, \rho + \kappa, \lambda_1 + \mu_1, \lambda_2 + \mu_2, \cdots).\]
Clearly, $\Psi$ is surjective and $\ker\Psi = \Img(1 - \Phi)$. Therefore, by Theorem~\ref{k-theory.theorem},
    \[K_0(\ucsalgshift) \cong \operatorname{coker}(1-\Phi) \cong \bigoplus_{j\in\mathbb{N}}\mathbb{Z}.\]
\end{example}

\section{The dynamical structure of the C*-algebra of a subshift}\label{dynamical}
Let $\osf$ be a subshift. Two groupoid models for the purely algebraic version of the subshift algebra, $\ualgshift$, are given in \cite{BCGW22} (a transformation groupoid and a Deaconu-Renault groupoid). The same groupoids give groupoid C*-algebra realizations of the subshift C*-algebra $\ucsalgshift$. We focus, first, on the groupoid constructed in \cite[Section 5.2]{BCGW22}, which arises as the transformation groupoid induced by the action of the free group  $\F$, generated by the alphabet $\alf$, on the Stone dual $\widehat{\TCB}$ of the Boolean algebra $\TCB$, defined in Section~\ref{unital}. The topology on $\widehat{\TCB}$ has a basis given by the sets $O_A=\{\xi\in \widehat{\TCB}: A\in \xi\}$, where $A\in\TCB$. This groupoid is defined as
\[ \F\ltimes_{\varphi} \HTCB = \{(\xi,t,\eta)\in \HTCB\times\F\times \HTCB: \eta\in V_{t^{-1}} \text{ and } \xi = \varphih_t (\eta) \},\]
where $V_{\beta\alpha^{-1}}=O_{C(\alpha,\beta)}$, and 
\[\varphih_{\alpha\beta^{-1}}(\xi)=\{A\in\TCB:r(B,\beta)\scj r(A,\alpha)\text{ for some }B\in\xi\}\] 
for every $\xi\in V_{\beta\alpha^{-1}}$, for every $\alpha,\beta\in\lang$ such that $\alpha \beta^{-1}$ is in reduced form, and $V_t= \emptyset$ if $t\neq\beta\alpha^{-1}$ for some $\alpha,\beta\in\lang$.

Multiplication and inverses in the groupoid are defined by
\[(\xi,s,\eta)(\eta,t,\gamma) = (\xi,st,\gamma) \text{ and } (\xi,t,\eta)^{-1}=(\eta,t^{-1},\xi),\]
respectively. 
Note that $\F\ltimes_{\varphi} \HTCB$ is an ample groupoid since the topology on $\HTCB$ has a basis of compact-open sets. For more details on this construction, see \cite[Section 5.2]{BCGW22}. We denote the isotropy group bundle of $\F\ltimes_{\varphih}\HTCB$ by $\operatorname{Iso}(\F\ltimes_{\varphih}\HTCB)$, and its interior by $\operatorname{Iso}(\F\ltimes_{\varphih}\HTCB)^0$. Recall that $\F\ltimes_{\varphih}\HTCB$ is \emph{effective} if $\operatorname{Iso}(\F\ltimes_{\varphih}\HTCB)^0= (\F\ltimes_{\varphih}\HTCB)^{(0)}$.
We say $\F\ltimes_{\varphih}\HTCB$ is \emph{topologically principal} if the set of units with trivial isotropy is dense in $(\F\ltimes_{\varphih}\HTCB)^{(0)}$. Let $C^*(\F\ltimes_{\varphih}\HTCB)$ denote the groupoid C*-algebra of $\F\ltimes_{\varphih}\HTCB$.

\begin{theorem}\label{iso.partial}
    Let $\osf$ be a subshift. Then $\ucsalgshift\cong C^*(\F\ltimes_{\varphih}\HTCB)$.
\end{theorem}

\begin{proof}
    One the one hand, $\TCA_\C(\osf)$ is a core subalgebra of $\ucsalgshift$, by Remark~\ref{r:core_subalg_univ}. On the other hand, the proof of \cite[Proposition~6.7]{BenGroupoid} implies that $\CA_\C(\F\ltimes_{\varphih}\HTCB)$ is a core subalgebra of $C^*(\F\ltimes_{\varphih}\HTCB)$. The result then follows from Proposition~\ref{core iso}.
\end{proof}

\begin{definition}
    Let $P\scj\lang$ be a finite set. The follower set of $P$ is defined as the set $F_P:=\bigcap_{\alpha\in P}F_\alpha$. We say that $\osf$ satisfies condition (L) if for every finite set $P\scj\lang$ and every $\gamma\in\lang$ such that $\gamma^{\infty}\in F_P$, there exists another element in $F_P$ other than $\gamma^{\infty}$.
\end{definition}

The map $\iota:\osf\to\HTCB$ given by $\iota(x)=\{A\in\TCB:x\in A\}$ is injective with dense image \cite[Proposition 5.17]{BCGW22}.

\begin{lemma}\label{lemox}
    If $\{x\}\in\TCB$ for some $x\in\osf$, then $O_{\{x\}}=\{\iota(x)\}$.
\end{lemma}

\begin{proof}
    In the case that $\{x\}\in\TCB$, we have that $\uparrow\{x\}\in\HTCB$. Hence the only element of $O_{\{x\}}$ is $\uparrow\{x\}=\iota(x)$.
\end{proof}

\begin{definition}
A topological partial action $(\{X_t,\theta_t\})_{t\in G}$ is topologically free if for every $e \neq g \in G$, the set $\{x \in X_{g^{-1}}: \theta_g(x) \neq x\}$ is dense in $X_{g^{-1}}$.
\end{definition}

\begin{theorem}\label{thm:(L)}
Let $\osf$ be a subshift over an arbitrary alphabet $\alf$. The following are equivalent:
\begin{enumerate}[(i)]
    \item\label{i:(L)} $\osf$ satisfies condition (L);
    \item\label{i:ev.per} $\{x\}\notin\TCB$ for every eventually periodic point $x\in\osf$;
    \item\label{i:eff} $\F\ltimes_{\varphih}\HTCB$ is effective;
    \item \label{i:top.free} the partial action $\varphih$ is topologically free. 
\end{enumerate}
Moreover, if the alphabet $\alf$ is countable, then the conditions above are also equivalent to 
\begin{enumerate}[(i)]
    \setcounter{enumi}{4}
    \item\label{i:top.princ} $\F\ltimes_{\varphih}\HTCB$ is topologically principal.
\end{enumerate}
\end{theorem}

\begin{proof}
    \eqref{i:(L)}$\Rightarrow$\eqref{i:ev.per} Suppose that for an eventually periodic point $\alpha\beta^{\infty}\in\osf$ we have that $\{\alpha\beta^{\infty}\}\in\TCB$. We will prove that there exists $P\scj\lang$ finite such that $F_P=\{\beta^\infty\}$ so that $\osf$ does not satisfy condition (L). The idea for the proof is similar to what is done in part of the proof of \cite[Proposition~7.10]{GillesDanie}. Since $\{\alpha\beta^{\infty}\}$ is a singleton in $\TCB$ it must be of the form
    \[\{\alpha\beta^{\infty}\}=C(\alpha_1,\beta_1)\cap\ldots\cap C(\alpha_n,\beta_n)  \cap C(\mu_1,\nu_1)^c\cap \ldots \cap C(\mu_m,\nu_m)^c.\]
    Note that for any $i\in\nn$, we have that, $r(\{\alpha\beta^{\infty}\},\alpha\beta^i)=\{\beta^{\infty}\}$. By applying \eqref{eq:rel.range} for sufficient large $i$, we see that $\{\beta^{\infty}\}$ is of the form
    \[\{\beta^{\infty}\}=F_{\gamma_1}\cap\cdots\cap F_{\gamma_k}\cap F_{\delta_1}^c\cap\cdots\cap F_{\delta_l}^c.\]
    Note that for $j\in\nn$ large enough, we have that $\delta_1\beta^j,\ldots,\delta_l\beta^j\notin\lang$. Taking the relative range with respect to $\beta^j$ and using \eqref{eq:rel.range}, we then see that
    \[\{\beta^{\infty}\}=F_{\gamma_1\beta^j}\cap\cdots\cap F_{\gamma_k\beta^j},\]
    from where it follows that $\osf$ does not satisfy condition (L).

    \eqref{i:ev.per}$\Rightarrow$\eqref{i:(L)} Suppose that $\osf$ does not satisfy condition (L). In this case there exist $P\scj\lang$ finite and $\gamma\in\lang$ such that $F_P=\{\gamma^{\infty}\}$. Since $F_P\in\TCB$, we found a periodic point $x$ such that $\{x\}\in\TCB$.

    For the equivalence \eqref{i:ev.per}$\Leftrightarrow$\eqref{i:eff}, we note that if $\varphih_t(\xi)=\xi$ for some $t\in\F\setminus\{\eword\}$, then $t=\alpha\beta\alpha^{-1}$ or $t=\alpha\beta^{-1}\alpha^{-1}$, where $\alpha,\beta\in\lang$ are such that $\pi(\xi)=\alpha\beta^\infty$ (see \cite[Remark 6.11]{GillesDanie}). Moreover, $\varphih_{t^{-1}}(\xi)=\xi$,  so we can consider only the case $t=\alpha\beta\alpha^{-1}$.

    \eqref{i:ev.per}$\Rightarrow$\eqref{i:eff} Suppose that $\F\ltimes_{\varphih}\HTCB$ is not effective. In this case, there exists $\alpha,\beta\in\lang$ with $|\beta|\geq 1$ and $A\in\TCB$ such that $\emptyset\neq A\scj C(\alpha\beta,\alpha)$ and $O_A\times\{\alpha\beta\alpha^{-1}\}\times O_A\scj\operatorname{Iso}(\F\ltimes_{\varphih}\HTCB)$.  Note that for every $x\in A$, we have $\iota(x)\in O_A$ and if $x$ is not of the form $\alpha\beta^{\infty}$, then $(\varphih_{\alpha\beta\alpha^{-1}}(\iota(x)),\alpha\beta\alpha^{-1},\iota(x))\notin \operatorname{Iso}(\F\ltimes_{\varphih}\HTCB)$. Hence $\{\alpha\beta^{\infty}\}=A\in\TCB$.

    \eqref{i:eff}$\Rightarrow$\eqref{i:ev.per} Suppose that, for $x=\alpha\beta^\infty$, we have that $\{x\}\in\TCB$. By Lemma~\ref{lemox}, $O_{\{x\}}=\{\iota(x)\}$.  By \cite[Proposition 5.3]{BCGW22}, there is a partial action $\tauh$ 
 of $\F$ on $\osf$ such that $\tauh_{\alpha\beta\alpha^{-1}}(x) = \alpha\beta^\infty =x$. By \cite[Proposition 5.17]{BCGW22}, the map $\iota:\osf\to\HTCB$ is equivariant with respect to $\varphih$ and $\tauh$. Hence $\varphih_{\alpha\beta\alpha^{-1}}(\iota(x))=\iota(x)$, so that $O_{\{x\}}\times\{\alpha\beta\alpha^{-1}\}\times O_{\{x\}}\scj\operatorname{Iso}(\F\ltimes_{\varphih}\HTCB)$. It follows that $\F\ltimes_{\varphih}\HTCB$ is not effective. 
    
    \eqref{i:eff}$\Leftrightarrow$\eqref{i:top.free} Notice that our definition of topologically free partial actions coincides with the definition of effective partial actions given in \cite{MR4009562}. The result now follows from \cite[Propositions~7.3]{MR4009562}.

    In the case that the alphabet is countable, the free group $\F$ is also countable, so \eqref{i:top.free}$\Leftrightarrow$\eqref{i:top.princ} follows from \cite[Corollary~2.5]{XinLi}.
\end{proof}

Recall the definition of continuous orbit equivalence of partial actions:

\begin{definition}\cite[Definition~2.6]{XinLi}
Partial actions $(\{X_g,\theta_g\})_{g\in G}$ and $(\{Y_h,\psi_h\})_{h\in H}$ are called continuously orbit equivalent if there exists a homeomorphism $\varphi: X \rightarrow Y$ and continuous maps $a: \: \bigcup_{g \in G} \{g\} \times X_{g^{-1}} \to H$, $b: \: \bigcup_{h \in H} \{h\} \times Y_{h^{-1}} \to G$ such that
$$  \varphi(\theta_g(x)) = \psi_{a(g,x)}(\varphi(x)), \text{ and }
  \varphi^{-1}(\psi_h(y)) = \theta_{b(h,y)}(\varphi^{-1}(y)).
$$
\end{definition}

Next, we characterize continuous orbit equivalence of the partial actions associated with a pair of subshifts that satisfy Condition~(L), in terms of their associated groupoids and C*-algebras.

\begin{theorem}\label{diadesol}
Let $\osf$ and $\osfY$ be subshifts over countable alphabets satisfying condition (L). Let $\varphih_\osf$ and $\varphih_\osfY$ be the corresponding partial actions of $\F_\osf$ and $\F_\osfY$ on $\HTCB_\osf$ and $\HTCB_\osfY$, respectively. Then, the following are equivalent:
\begin{enumerate}
\item[(i)] $\varphih_\osf$ and $\varphih_\osfY$ are continuously orbit equivalent,
\item[(ii)] the transformation groupoids $\F_\osf \ltimes \HTCB_\osf$ and $\F_\osfY \ltimes \HTCB_\osfY$ are isomorphic as topological groupoids,
\item[(iii)] there exists an isomorphism $\Phi: \ucsalgshift \rightarrow \ucsalgshiftY$ with $\Phi(C(\widehat{\mathcal {U}_X})) = C(\widehat{\mathcal {U}_Y})$.
\end{enumerate}
Moreover, (ii) implies (i) holds even if condition (L) is not satisfied.
\end{theorem}

\begin{proof}
    The equivalences follow from \cite[Theorem~2.7]{XinLi} and Theorem~\ref{thm:(L)}.
\end{proof}

We now consider the problem of finding invariants for topological conjugacy for OTW-subshifts and how it relates to the isomorphism of subshift C*-algebras. We do this by adding some equivalent conditions to \cite[Theorem~7.6]{BCGW22}. For this we use the Deaconu-Renault groupoid $\CG(\HTCB,\hsig)$ associated with $\osf$ (\cite{BCGW22}), where $\HTCB$ is the Stone dual of $\TCB$ as before, and $\hsig$ has domain  $\dom(\hsig):=\bigcup_{a\in\alf} V_a = \bigcup_{a\in\alf} O_{Z_a}$ and is defined by 
\[\hsig(\xi):=\varphih_{a^{-1}}(\xi),\]
for $\xi \in V_a$. See \cite[Section~6]{BCGW22} for more details. By Theorem~\ref{iso.partial} and \cite[Theorem~6.4]{BCGW22}, we have that $\ucsalgshift\cong C^*(\CG(\HTCB,\hsig))$.

Given an OTW-subshift $\osf^{OTW}$ and its corresponding subshift $\osf$, \cite[Proposition~6.3]{BCGW22} provides a surjective map $\pi:\HTCB\to\osf^{OTW}$ such that $\pi(\hsig(\xi))=\sigma(\pi(\xi))$ for all $\xi\in\dom(\hsig)$. Moreover, as in \cite[Remark~6.2]{BCGW22}, one can use Proposition~\ref{nublado} to show that $C(\osf^{OTW})$ is isomorphic to a *-subalgebra of $C(\HTCB)$ by sending $f\in C(\osf^{OTW})$ to $f\circ \pi$. We denote this *-subalgebra by $\pi^*(C(\osf^{OTW}))$.

We recall the definition of a conjugacy between OTW-subshifts.

\begin{definition}
    Let $ \osf^{OTW}$ and $ \osfY^{OTW}$ be OTW-subshifts over alphabets $\alf_1$ and $\alf_2$, respectively. A map $h:  \osf^{OTW}\rightarrow \osfY^{OTW}$ is a \emph{conjugacy} if it is a homeomorphism, commutes with the shifts and is length-preserving.
\end{definition}

In Theorem~\ref{stariso} we characterize conjugacy of OTW-subshifts. For that, the idea is to invoke  \cite[Theorem~3.1]{TopConjLocHom}. However, the shift map of an OTW-subshift is not a local homomorphism in general. For instance, a subshift over a finite alphabet is a local homeomorphism if and only if it is of finite type \cite[Proposition~2.5]{MishaRuy}.  Nevertheless, using \cite[Proposition~7.5 and Theorem~7.6]{BCGW22} we can rephrase the problem in terms of a conjugacy between the corresponding Deaconu-Renault system of their subshifts.

Before the next theorem, we briefly adapt a definition used in \cite[Theorem~3.1]{TopConjLocHom} to our needs. Given a subshift $\osf$ and $f\in C(\HTCB,\mathbb{Z})$, we define the \emph{weighted action} $\gamma^{\osf,f}:\mathbb{T}\curvearrowright \ucsalgshift$ on the generators by $\gamma^{\osf,f}_z(p_A)=p_A$ and $\gamma^{\osf,f}_z(s_a)=z^fs_a$ for all $A\in\TCB$, $a\in\alf$ and $z\in\mathbb{T}$, where $z^f$ is to be interpreted as an element of $C(\HTCB)$ seen as a subalgebra of $\ucsalgshift$. More specifically, $z^f:\HTCB\to\mathbb{C}$ is defined by $z^f(\xi)=z^{f(\xi)}$ for each $z\in\mathbb{T}$ and $\xi\in\HTCB$.

\begin{theorem}\label{stariso}
    Let $\osf$ and $\osfY$ be subshifts over countable alphabets and $h:  \osf^{OTW}\rightarrow \osfY^{OTW}$ a homeomorphism. The following are equivalent:
    \begin{enumerate}[(i)]
    \item\label{iconj} $h$ is a conjugacy.
    \item\label{ihh} There exists a homeomorphism $\hh:\HTCB_\osf\to \HTCB_\osfY$ such that $h\circ\pi_\osf=\pi_\osfY\circ\hh$, $\hh(\dom(\hsig_\osf))=\dom(\hsig_\osfY)$ and $\hh\circ\hsig_\osf=\hsig_\osfY\circ\hh|_{\dom(\hsig_\osf)}$.
    \item\label{iC*} There exists a *-isomorphism $\Phi: \ucsalgshift \rightarrow \ucsalgshiftY$ satisfying the following conditions:
    \begin{enumerate}[(a)]
        \item $\Phi(\pi_\osf^*(C(\osf^{OTW}))) = \pi_\osfY^*(C(\osfY^{OTW}))$,
        \item $\Phi(f\circ\pi_\osf)=f\circ h^{-1}\circ\pi_{\osfY}$ for all $f\in C(\osf^{OTW})$,
        \item there exists a homeomorphism $\widetilde{h}:\HTCB_\osf\to \HTCB_\osfY$ such that $\Phi\circ\gamma_z^{\osf,g\circ\widetilde{h}}=\gamma^{\osfY,g}_z\circ\Phi$ for all $g\in C(\osfY,\mathbb{Z})$ and $z\in\mathbb{T}$.
    \end{enumerate}
    \end{enumerate}
\end{theorem}

\begin{proof}
    The equivalence between \eqref{iconj} and \eqref{ihh} is proved in \cite[Theorem~7.6]{BCGW22}.

    \eqref{ihh}$\Rightarrow$\eqref{iC*} By \cite[Proposition~3.12(i)]{TopConjLocHom}, there exists a *-homomorphism $\Phi:\ucsalgshift\to\ucsalgshiftY$ such that $\Phi(C(\widehat{\mathcal {U}_X})) = C(\widehat{\mathcal {U}_Y})$, $\Phi(f)=f\circ\hh^{-1}$ for all $f\in C(\widehat{\mathcal {U}_X})$ and (c) of item \eqref{iC*}. Items (a) and (b) follow from the hypothesis that $h\circ\pi_\osf=\pi_\osfY\circ\hh$.

    \eqref{iC*}$\Rightarrow$\eqref{ihh} By \cite[Proposition~3.12(ii)]{TopConjLocHom}, there exists a homeomorphism $\hh:\HTCB_\osf\to \HTCB_\osfY$ such that $\hh(\dom(\hsig_\osf))=\dom(\hsig_\osfY)$, $\hh\circ\hsig_\osf=\hsig_\osfY\circ\hh|_{\dom(\hsig_\osf)}$ and $\Phi(g)=g\circ\hh^{-1}$ for all $g\in C(\widehat{\mathcal {U}_X})$. To prove the last part of \eqref{ihh}, take any $f\in C(\osf^{OTW})$ and $\xi\in\widehat{\mathcal {U}_X}$. On the one hand, \[\Phi(f\circ\pi_\osf)(\hh(\xi))=f(\pi_\osf(\hh^{-1}(\hh(\xi))))=f(\pi_\osf(\xi)).\] On the other hand, by hypothesis, we have that
    \[\Phi(f\circ\pi_\osf)(\hh(\xi))=f(h^{-1}(\pi_\osfY(\hh(\xi)))).\]
    Since $f$ was arbitrary and $h$ is bijective, we obtain that $h(\pi_\osf(\xi))=\pi_\osfY(\hh(\xi))$. And, since $\xi$ was arbitrary, we have that $h\circ\pi_\osf=\pi_\osfY\circ\hh$.
\end{proof}

For a countable alphabet $\alf$, endow $\prod_{n\in\nn} \alf$ with the product metric 
\begin{equation}\label{SBC} 
d((x_n),(y_n))= \frac{1}{2^k},
\end{equation} 
where $k$ is such that $x_i=y_i$ for $i=0,\ldots k-1$ and $x_k\neq y_k$. This metric induces the product topology on $\prod \alf_{n\in\nn}$. Viewing $\prod_{n\in\nn} \alf$ as topological space with topology induced by this metric, a subshift means a closed and invariant subset $\osf\subseteq \prod_{n\in\nn}\alf$. Thus, any subshift is complete in this metric. Notice that if $\alf$ is infinite then $\osf\subseteq \osf^{OTW}$ is, in general, not complete in the metric of $\osf^{OTW}$, which shows these two metrics are not strongly equivalent. In particular, a conjugacy of subshifts $\osf_1$ and $\osf_2$ does not automatically extend to a conjugacy of $\osf^{OTW}_1$ and $\osf^{OTW}_2$. However, we do have the following extension result.

\begin{theorem}\label{aguaceiro} Let $\osf_1$ and $\osf_2$ be subshifts over countable alphabets. Abusing the notation, let $(\osf_1,\sigma)$ and $(\osf_2,\sigma)$ be the subshifts equipped with the metric $d$ in (\ref{SBC}). If $(\osf_1,\sigma)$ is isometrically conjugate to $(\osf_2,\sigma)$ then the respective OTW subshifts, $\osf^{OTW}_1$ and $\osf^{OTW}_2$, are conjugate. Moreover, if $\osf^{OTW}_1$ and $\osf^{OTW}_2$ are conjugate then $(\osf_1,\sigma)$ and $(\osf_2,\sigma)$ are conjugate.   
\end{theorem}

\begin{proof}
Suppose that $h:\osf_1 \rightarrow \osf_2$ is an isometric conjugacy. We define a map $h^{OTW}$ from $\osf^{OTW}_1$ to $\osf^{OTW}_2$ as follows. 

Let $x\in \osf^{OTW}_1$. If $|x|=\infty$,  define $h^{OTW}(x) = h(x)$. If $|x|< \infty$, then there exists $y\in\osf$ such that $xy\in\osf$. We then define $h^{OTW}(x)$ as the first $|x|$ entries of $h(xy)$. To show that it is well-defined, we need to prove that $h^{OTW}(x)$ is independent of the choice of $y$ and that $h^{OTW}(x)\in \osf_2^{OTW}$. Let $z\in\osf$ be such that $xz\in\osf$ and note that $d(xy,xz)\leq \frac{1}{2^{|x|}}$. Because $h$ is an isometry $d(h(xy),h(xz))\leq \frac{1}{2^{|x|}}$, so that $h(xy)$ and $h(xz)$ coincide in the first $|x|$ coordinates. Next, we prove that $h^{OTW}(x)\in \osf_2^{OTW}$. Let $\{a_n\}_{n\in\nn}$ be an infinite subset of distinct elements of $\alf_1$ such that for each $n\in\nn$, there exists $y_n$ with $xa_ny_n\in\osf$ (see the definition of $\osf_1^{fin}$ in Section \ref{symbolic}). Note that if $n\neq m$, then $d(xa_ny_n,xa_my_m)=\frac{1}{2^{|x|}}$. Again, using that $h$ is an isometry, we obtain that $d(h(xa_ny_n),h(xa_my_m))=\frac{1}{2^{|x|}}$, so that $h(xa_ny_n)=h(x)b_nz_n$ for a infinite subset of elements $\{b_n\}_{n\in\nn}$ of $\alf_2$ and some $z_n\in\osf_2$ for each $n\in\nn$.

Clearly $h^{OTW}$ is length preserving and it is shift-commuting because $h$ is. Notice that $h^{-1}$ is also an isometry that commutes with the shift. Applying the same construction as above, we build $(h^{-1})^{OTW}$, which is easy to see is the inverse of $h^{OTW}$.

It remains to show that $h^{OTW}$ is a homeomorphism. By symmetry, it is enough to show that it is open. Let $\alpha\in\CL_{\osf_1}$. Even though it may not be the case that $\alpha\in\osf_1^{OTW}$, the same argument above using that $h$ is an isometry shows that there exists $\beta\in\CL_{\osf_2}$ such that $|\beta|=|\alpha|$ and $h(Z_{\alpha})= Z_{\beta}$. Then $\CZ(\alpha)=\CZ(\beta)$. Because $h^{OTW}$ is a bijection, it preserves set operations, so that if $F\scj\alf_1$ is finite, then $h(\CZ(\alpha,F))=\CZ(\beta,G)$ for some $G\scj\alf_2$ finite. Hence $h$ is open, since sets of the form $\CZ(\alpha,F)$ generate the topology on $\osf_1^{OTW}$.

For the converse statement, note that a conjugacy of OTW subshifts is, by definition, length preserving. Therefore, since  $(\osf^{OTW})^{\text{inf}} = \osf$ for any subshift $\osf$, a conjugacy of OTW subshifts restricts to a conjugacy of subshifts with the product topology. 
\end{proof}

\begin{remark}
    Recall that for OTW-subshifts the definition of a conjugacy includes the length preserving condition, in addition to the usual requirement that the map is a shift commuting homeomorphism. The difference between the definitions of conjugacy for subshifts with the metric $d$ given in \eqref{SBC} and OTW-subshifts may also be understood from the shift commuting condition, which is equivalent to the map being a sliding block code for subshifts with the metric given in \eqref{SBC}, and a generalised sliding block code for OTW-subshifts, see \cite{GSS, MR3918205, MR3620337}.
\end{remark}

\begin{theorem}\label{invariant}
    If $\osf_1$ and $\osf_2$ are isometrically conjugate subshifts over countable alphabets, then there exists a diagonal-preserving gauge-invariant *-isomorphism between the C*-algebras $\TCO_{\osf_1}$ and $\TCO_{\osf_2}$.
\end{theorem}

\begin{proof}
    It follows from Theorems~\ref{stariso} and \ref{aguaceiro}.
\end{proof}

\begin{remark}
    The K-theory of the corresponding C*-algebras of subshifts can be used as an invariant for isometric conjugacy. For instance, the subshifts considered in Examples~\ref{chuva} and \ref{tripla} are not isometrically conjugate. 
\end{remark}

\bibliographystyle{abbrv}
\bibliography{ref}

\end{document}